\documentclass[final]{l4dc2025}

\title[Physics-informed Gaussian Processes as Linear Model Predictive Controller]{Physics-informed Gaussian Processes as Linear Model Predictive Controller}
\usepackage{times}
\usepackage{amsmath}
\usepackage{amsfonts}       
\usepackage{graphicx}
\usepackage{hyperref}

\usepackage{acro}

\usepackage{tikz}
\usepackage{pgfplots}
 \usetikzlibrary{
    angles,
 	arrows,
 	arrows.meta,
 	shapes,
 	shapes.symbols,
 	shapes.callouts,
 	shapes.geometric,
 	arrows,
 	matrix,
 	backgrounds,
 	positioning,
 	plotmarks,
 	calc,
 	patterns,
 	matrix,
 	decorations.pathreplacing,
 	decorations.pathmorphing,
 	decorations.text,
 	decorations.shapes,
 	decorations.fractals,
 	decorations.markings,
 	spy,
    quotes
 }
\usepgfplotslibrary{fillbetween}
\usepgfplotslibrary{statistics}
\usepgfplotslibrary{groupplots}

\pgfplotsset{
  /pgfplots/confidence box/.style 2 args={
    legend image code/.code={
        \definecolor{steelblue31119180}{RGB}{31,119,180}
        \draw[steelblue31119180,no markers, fill=steelblue31119180, opacity=0.5]
        plot coordinates {
        (-0.1cm,-0.1cm)
        (-0.1cm,0.2cm)
        (0.5cm,0.2cm)
        (0.5cm,-0.1cm)
        (-0.1cm,-0.1cm)
      }
      node[rectangle]{};
    }
  }
}

\usepackage{tikz}
\usepackage{graphicx}
\usetikzlibrary{shapes,arrows}
\usepackage{pgfplots}
\pgfplotsset{compat=newest}
\usepackage{amsmath, amsfonts}
\usepackage[noend]{algpseudocode}
\usepackage{algorithm}
\acsetup{patch/maketitle = false}
\usepackage{todonotes}
\usepackage{dsfont}
\DeclareAcronym{aic}{
    short = AIC,
    long = Akaike Information Criterion 
}

\DeclareAcronym{aicc}{
    short = AICc,
    long = Akaike Information Criterion correction 
}

\DeclareAcronym{bfgs}{
    short = BFGS,
    long = Broyden-Fletcher-Goldfarb-Shanno algorithm
}

\DeclareAcronym{bic}{
    short = BIC,
    long = Bayesian Information Criterion 
}

\DeclareAcronym{cks}{
    short = CKS,
    long = Compositional Kernel Search 
}

\DeclareAcronym{dl}{
    short = DL,
    long = Deep Learning 
}

\DeclareAcronym{gp}{
    short = GP,
    long = Gaussian Process,
    long-plural-form = Gaussian Processes 
}

\DeclareAcronym{hmc}{
    short = HMC, 
    long = Hamiltonian Monte Carlo
}

\DeclareAcronym{kl}{
    short = KL, 
    long = Kullback-Leibler
}

\DeclareAcronym{ks}{
    short = KS,
    long = Kernel Search 
}

\DeclareAcronym{lfm}{
    short = LFM, 
    long = Latent Force Model 
}

\DeclareAcronym{lodegp}{
    short = LODE-GP, 
    long = Linear Ordinary Differential Equation Gaussian Process,
    long-plural-form = Linear Ordinary Differential Equation Gaussian Processes
}

\DeclareAcronym{lti}{
    short = LTI, 
    long = Linear Time Invariant 
}

\DeclareAcronym{map}{
    short = MAP, 
    long = Maximum A Posteriori 
}

\DeclareAcronym{mc}{
    short = MC, 
    long = Monte Carlo 
}

\DeclareAcronym{mcmc}{
    short = MCMC, 
    long = Markov Chain Monte Carlo 
}

\DeclareAcronym{mll}{
    short = MLL, 
    long = Marginal Log Likelihood 
}

\DeclareAcronym{ml}{
    short = ML, 
    long = Machine Learning
}

\DeclareAcronym{mpc}{
    short = MPC, 
    long = Model Predictive Control 
}

\DeclareAcronym{nn}{
    short = NN, 
    long = Neural Network 
}

\DeclareAcronym{nuts}{
    short = NUTS, 
    long = No U-Turn Sampler 
}

\DeclareAcronym{ode}{
    short = ODE, 
    long = Ordinary Differential Equation 
}

\DeclareAcronym{rmse}{
    short = rmse, 
    long = Root Mean Squared Error 
}

\DeclareAcronym{se}{
    short = SE, 
    long = Squared Exponential 
}

\DeclareAcronym{sgd}{
    short = SGD, 
    long = Stochastic Gradient Descent 
}

\DeclareAcronym{skc}{
    short = SKC, 
    long = Structured Kernel Composition 
}

\DeclareAcronym{snf}{
    short = SNF, 
    long = Smith Normal Form 
}

\DeclareAcronym{rkhs}{
    short = RKHS, 
    long = Reproducing Kernel Hilbert Space 
}

 \coltauthor{\Name{Jörn Tebbe} \Email{joern.tebbe@th-owl.de}\\
  \Name{Andreas Besginow} \Email{andreas.besginow@th-owl.de}\\
  \Name{Markus Lange-Hegermann} \Email{markus.lange-hegermann@th-owl.de}\\
  \addr Institute Industrial IT - inIT, OWL University of Applied Sciences and Arts Lemgo, Germany}

\begin{document}
\maketitle
\begin{abstract}%
We introduce a novel algorithm for controlling linear time invariant systems in a tracking problem.
The controller is based on a Gaussian Process (GP) whose realizations satisfy a system of linear ordinary differential equations with constant coefficients.
Control inputs for tracking are determined by conditioning the prior GP on the setpoints, i.e.\ control as inference. 
The resulting Model Predictive Control scheme incorporates pointwise soft constraints by introducing virtual setpoints to the posterior Gaussian process.
We show theoretically that our controller satisfies open-loop stability for the optimal control problem by leveraging general results from Bayesian inference and demonstrate this result in a numerical example.
\end{abstract}

\section{Introduction}

Controlling industrial applications requires precise modeling and good control algorithms, which is often addressed via \ac{mpc} \citep{tebbe2023holistic, rawlings2017model}.
\ac{mpc} consists of a predictive model and a control strategy.
The predictive model simulates the future system behavior for given control inputs and the optimization strategy produces the optimal input with respect to the given objective function and constraints.

Predictive models are usually \emph{first principle based} and/or \emph{data-driven}.
For example, \acp{gp} have emerged as a commonly applied data-driven method due to their excellent handling of both few datapoints and uncertainty quantification \citep{berkenkamp2015safe, hewing2018cautious, maiworm2021online}\footnote{For a comparison of \ac{gp}-based approaches to Physics-informed Neural Networks see e.g.\ \cite{harkonen2023gaussian}}.
The objective function may be of economic type \citep{bradford2018economic} or quadratic type \citep{qin2003survey}.
Linear \ac{mpc} for tracking defines the special case where the predictive model is linear and the objective is the squared Euclidean distance of the state to a desired setpoint and therefore quadratic \citep{limon2008mpc}.

The resulting optimization problem is a Quadratic Program (QP) whose solution is easily obtained.
A problem of this constrained optimization problem is, that it might be infeasible due to the initial point not satisfying the constraints \citep{kerrigan2000soft, krupa2024model}. 
A remedy to this is soft constrained \ac{mpc} where the optimization problem has no hard state and control constraints, but contains these constraints in its objective function.
Several works addressed this research path, including \cite{zeilinger2014soft, wabersich2021soft, gracia2024implementation}.
In this work we present a novel approach for such problems with input and state constraints.
We expand the recently introduced class of \acp{lodegp} \citep{besginow2022constraining}, a class of \acp{gp} strictly satisfying given linear \ac{ode} systems to an \ac{mpc} scheme. 
We solve the optimal control problem by using the current state and constraints as training data for the \ac{lodegp} and obtaining the control law directly from its posterior predictive distribution.
By using smooth kernel functions in our \ac{gp}, our method also produces smooth control functions, but other choices of kernels would be possible \citep{duvenaud2014automatic}. 
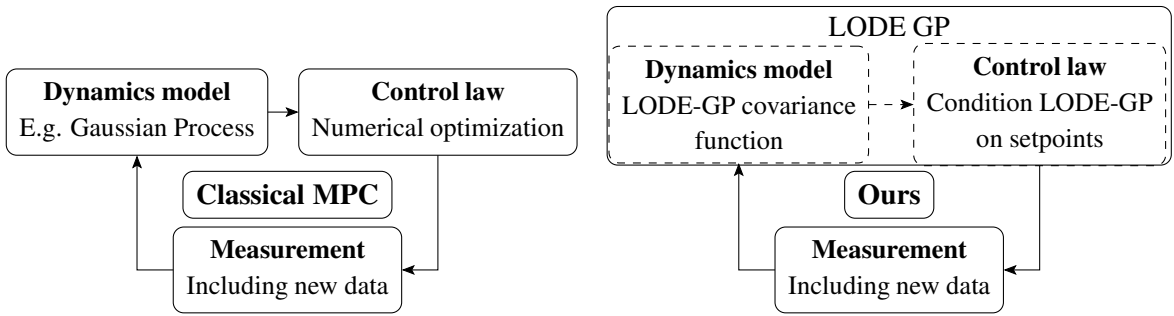
\begin{figure*}[t]
    \centering
    \definecolor{steelblue31119180}{RGB}{31,119,180}
\begin{tikzpicture}
\node[draw, rounded corners, inner sep=5pt, outer sep=0pt, align=center] (Dynamics_MPC) at (0.0,5.0){\small\textbf{Dynamics model} \\ \small E.g. Gaussian Process};
\node[draw, rounded corners, inner sep=5pt, outer sep=0pt, align=center] (Optimization_MPC) at (4.0,5.0){\small\textbf{Control law} \\ \small Numerical optimization};
\node[draw, rounded corners, inner sep=5pt, outer sep=0pt, align=center] (Measurement_MPC) at (2.0,2.9){\small \textbf{Measurement} \\ \small Including new data};
\node[draw, rounded corners, inner sep=5pt, outer sep=0pt, align=center](Classical MPC) at (2.0, 3.9){\textbf{Classical MPC}};
      
\node[draw, rounded corners, inner sep=5pt, outer sep=0pt, align=center, minimum width=7.5cm, minimum height=2.11cm] (LODE_GP) at (10.0,5.35){};
\node[anchor=north] at (LODE_GP.north) {LODE GP};
\node[draw, dashed, rounded corners, inner sep=5pt, outer sep=0pt, align=center] (Dynamics_GP) at (8.0,5.1){\small \textbf{Dynamics model} \\ \small LODE-GP covariance\\  \small function};
\node[draw, dashed, rounded corners, inner sep=5pt, outer sep=0pt, align=center] (Optimization_GP) at (12.0,5.1){\small \textbf{Control law} \\ \small Condition LODE-GP\\ \small on setpoints};
\node[draw, rounded corners, inner sep=5pt, outer sep=0pt, align=center] (Measurement_GP) at (10.0,2.9){\small \textbf{Measurement} \\ \small Including new data};
\node[draw, rounded corners, inner sep=5pt, outer sep=0pt, align=center](Classical MPC) at (10.0, 3.9){\textbf{Ours}};

\draw[arrows = {-Stealth[length=5pt, inset=1pt]}] (Dynamics_MPC) -- (Optimization_MPC);
\draw[arrows = {-Stealth[length=5pt, inset=1pt]}] (Optimization_MPC) |- (Measurement_MPC);
\draw[arrows = {-Stealth[length=5pt, inset=1pt]}] (Measurement_MPC) -| (Dynamics_MPC);

\draw[dashed, arrows = {-Stealth[length=5pt, inset=1pt]}] (Dynamics_GP) -- (Optimization_GP);
\draw[arrows = {-Stealth[length=5pt, inset=1pt]}] (Optimization_GP) |- (Measurement_GP);
\draw[arrows = {-Stealth[length=5pt, inset=1pt]}] (Measurement_GP) -| (Dynamics_GP);

\end{tikzpicture}
    \caption{Similarities between classical \ac{mpc} and our proposed method.
    While in classical \ac{mpc} (Left) a \ac{gp} or different surrogate model is only used for the dynamics, our approach (Right) applies also the synthesis of the control law \emph{directly in the \ac{gp}}.}
    \label{fig:figure_1}
\end{figure*}
Our approach differs from classical tracking \ac{mpc} by considering \emph{the union of dynamics and control law in one model}. 
This reduces a complex control task to simple posterior inference of a \ac{gp}, as we illustrate in Figure~\ref{fig:figure_1}. This approach, also known as Control as Inference (CAI), has already been used in stochastic optimal control problems and reinforcement learning \citep{levine2018reinforcement}.
Our approach is furthermore based on the behavioral approach to control \citep{willems1997introduction}, which does not necessarily distinguish inputs, state, and outputs and instead combines them as one matrix.
This allows application of computer algebra, with the potential of mixing inputs, states and outputs through base-change matrices.
This point of view allows to apply computer algebra \citep{oberst1990multidimensional,pommaret1999algebraic,zerz2000topics,chyzak2005effective,lange2013thomas,lange2020thomas}, which was also the motivation of the \ac{lodegp}.
In this paper, we endow the behavior, i.e.\ the set of admissible trajectories, with a probability distribution via the \acp{lodegp} and use the marginalization of this probability distribution for control.
This probabilistic behavioral approach allows to restrict our system by any form of data, e.g.\ at any time, masked data,  average values, or most importantly by requirements of future states.
We exploit the kernelized structure of the \ac{gp} to provide open-loop stability to the controlled system. This results from the fact that a \ac{gp} posterior converges to its prior in the absence of correlated training data. Moreover, we connect \acp{gp} with \ac{mpc} via CAI. This connection may open the field of control for \ac{mpc} for the vast methods of \acp{gp} or other kernel based methods including variational approaches for non-Gaussian likelihoods \citep{titsias2009variational}.

\emph{Notation:} We denote the concatenation of state $x \in \mathbb{R}^{n_x}$ and control $u \in \mathbb{R}^{n_u}$ as $z \in \mathbb{R}^{n_z}$. We use $\dot x$ as the derivative of $x(t)$ with respect to $t$, $x\leq y$ for vectors denotes $x_i\leq y_i$ in each entry, $\mathds{1}_A$ is the indicator function for $x \in A$, $\delta(x,x')$ denotes the Kronecker function and diag$(x_1,\dots,x_n)$ is a diagonal matrix containing $x_1,\dots, x_n$ on the diagonal.


\section{Problem formulation}
\label{sec:problem}
Consider the controllable linear \ac{ode} based system with state $x \in \mathbb{R}^{n_x}$ and control input $u \in \mathbb{R}^{n_u}$
\begin{align}\label{eq:linear_system}
    \dot x &= A x + Bu 
\end{align}
where $A \in \mathbb{R}^{n_x \times n_x}$ and $B \in \mathbb{R}^{n_x \times n_u}
$ are system and control matrices respectively \citep{rawlings2017model}.
The system state and control input are subject to constraints
\begin{align}\label{eq:box_constraints_x}
    x_{\min} &\leq x(t) \leq x_{\max} \\
    u_{\min} &\leq u(t) \leq u_{\max}\label{eq:box_constraints_u}
\end{align}
for $t \in [t_0, t_T]$. 
The general tracking control task is the minimization of the difference between a given reference $x_{\text{ref}} \in \mathbb{R}^{n_x}$ and states $x$ over a given time horizon $[t_0, t_T]$ 
\begin{equation}\label{eq:integral_formulation}
    \min_{u(t)} \int_{t_0}^{t_T} ( x_{\text{ref}} - x(t) )^2 dt
\end{equation}
with the constraints \eqref{eq:linear_system}, \eqref{eq:box_constraints_x}, \eqref{eq:box_constraints_u} and given initial point $x(t_0) = x_0$.
The tracking control task at discrete timesteps is an approximation of \eqref{eq:integral_formulation} and minimizes the distance of state
$x_{\text{ref}}$ and state $x(t)$ to find the minimal error control solution of
\begin{subequations}
\begin{align} 
    \min_{u(t)} &\sum_{i=0}^{T} (x_{\text{ref}} - x(t_i))^2 + \Vert u \Vert\label{eq:mpc:obj}\\
    \text{s.t. } \dot x &= Ax + Bu, \label{eq:mpc:con:ode}\\
    x(t_0) &= x_0, \label{eq:mpc:con:init}\\
    x_{\min} &\leq x(t) \leq x_{\max} \quad \forall t \in [t_0,t_T], \label{eq:mpc:con:x}\\
    u_{\min} &\leq u(t) \leq u_{\max} \quad \forall t \in [t_0,t_T] \label{eq:mpc:con:u}.
\end{align}
\end{subequations}
We will present an approach in this paper which will approximate the solution of this optimization problem by giving a reference point $x_{\text{ref}}$ and forcing the resulting solution $x(t)$ and $u(t)$ to provide smooth behavior.
Note that the constraints \eqref{eq:mpc:con:ode} and \eqref{eq:mpc:con:init} must be considered as hard constraints, while the remaining can be incorporated as soft constraints in order to guarantee feasibility of the optimization problem.

\section{Preliminaries}
\label{sec:preliminaries}
\subsection{Gaussian Processes}
A \acf{gp} \citep{rasmussen2006gaussian} $g(t) \sim \mathcal{GP}(\mu(t), k(t, t'))$ is a stochastic process with the property that all $g(t_1), \ldots, g(t_n)$ are jointly Gaussian.
Such a \ac{gp} is fully characterized by its mean $\mu(t)$ and covariance function $k(t, t')$.
By conditioning a \ac{gp} on a noisy dataset $\mathcal{D} = \{(t_1, z_1), \ldots ,(t_n, z_n)\}$ we have the posterior \ac{gp} defined as

\begin{equation}\label{eq:gaussian_process_posterior_distribution}
    \begin{aligned}
        \mu^* &= \mu(t^*) + K_*^\top(K + \sigma_n^2 I)^{-1}(z - \mu(t)) \\
        k^* &= K_{**} - K_*^\top(K + \sigma_n^2 I)^{-1} {K_*}
    \end{aligned}
\end{equation}
with covariance matrices $K = (k(t_i, t_j))_{i,j} \in \mathbb{R}^{n \times n}$, $K_* = (k(t_i, t^*_j))_{i, j} \in \mathbb{R}^{n \times m}$ and $K_{**} = (k(t^*_i, t^*_j))_{i, j} \in \mathbb{R}^{m \times m}$ for predictive positions $t^* \in \mathbb{R}^m$ with noise variance $\sigma_n^2$.
This is the most common way of applying \acp{gp} in control theory for regression analysis on time series data. 

Additionally, \acp{gp} can be parameterized in terms of hyperparameters $\theta$, which include the noise variance $\sigma_n^2$. The noise variance does not have to be constant (homoscedastic), but can be input dependent (heteroscedastic), i.e. $\sigma_n^2(t) \in \mathbb{R}^{n_z}$. The noise variance describes the noise we expect on a given datapoint $(t_i,z_i)$.
Additional hyperparameters are commonly introduced via its covariance function, for example the \ac{se} covariance function often includes signal variance $\sigma_f^2$ and smoothness parameter $\ell^2$: 
\begin{equation}\label{eq:SE_kernel}
    k_{\text{SE}}(t, t') = \sigma_f^2\exp\left(-\frac{(t-t')^2}{2\ell^2}\right)
\end{equation}

These hyperparameters are trained by maximizing the \acp{gp} \ac{mll}:
\begin{equation} \label{eq:gp:MLL}
    \log p(z|t) = -\frac{1}{2}z^\top\left(K+\sigma_n^2I\right)^{-1}z - \frac{1}{2}\log \left( \det \left(K+\sigma_n^2I\right) \right)
\end{equation}
where $I$ is the identity and constant terms are omitted. We obtain a quadratic type error term combined with a regularization term based on the determinant of the regularized kernel matrix.


\subsection{Linear Ordinary Differential Equation GPs}\label{sec:LODE_GP}
The class of \acp{gp} is closed under linear operations\footnote{This property holds true for almost all relevant linear cases in control theory. For more details see \citep{harkonen2023gaussian, matsumoto2024images}.}, i.e.\ applying a linear operator $\mathcal{L}$ to a \ac{gp} $g$ as $\mathcal{L}g$ is again a \ac{gp}.
This ensures that realizations of the \ac{gp} $\mathcal{L}g$ lie in the image of the linear operator $\mathcal{L}$, in addition to the \ac{gp} $g$ \citep{jidling2017linearly,langehegermann2018algorithmic}.

We demonstrate the procedure from \citep{besginow2022constraining} for constructing so-called \acp{lodegp} --- \acp{gp} that strictly satisfy the underlying system of linear homogenuous ordinary differential equations --- using the following unstable system, with two integrators of the control function.
\begin{align}
	\dot x = \begin{pmatrix}
	0 & 1 \\
	1 & 1 
	\end{pmatrix} x + \begin{pmatrix}
	0 \\ 1
	\end{pmatrix}
	u
\end{align}
We subtract $\dot xI$ and combine state $x$ and input $u$ by stacking it in one variable $z$ to reformulate the system as:
\begin{align}\label{eq:harmonic_oscillator_diffeq}
	0 = H \cdot z =  \begin{pmatrix}
	    -\partial_t  & 1 & 0 \\
        1 & 1 - \partial_t  & 1 
	\end{pmatrix} \begin{pmatrix}
	x_1 \\ x_2 \\ u
	\end{pmatrix}
\end{align}
We can algorithmically factor $H$ into three matrices such that $Q\cdot H\cdot V = D$, with $D \in \mathbb{R}[\partial_t]^{n_x \times n_z}$ the Smith Normal Form and $Q \in \mathbb{R}[\partial_t]^{n_x \times n_x}, V \in \mathbb{R}[\partial_t]^{n_z \times n_z}$ invertible \citep{smith1862systems, newman1997smith}.
All matrices belong to the polynomial ring $\mathbb{R}[\partial_t]$ i.e.\ containing polynomials of $\partial_t$.
For the system in Equation~\eqref{eq:harmonic_oscillator_diffeq} the application of an algorithm to find the Smith Normal Form results in the following $D$, $Q$ and $V$:
\begin{equation}
    D = \begin{pmatrix}
        1 & 0 & 0\\
        0 & 1 & 0
    \end{pmatrix},
    Q =\begin{pmatrix}
        0 & 1\\
        -1 & -\partial_t
	\end{pmatrix},
    V =\begin{pmatrix} 
	1 & 0     & 1 \\
	0 & -1    & \partial_t \\
    0 & -\partial_t-1  & \partial_t^2 + \partial_t - 1
	\end{pmatrix}
\end{equation}
We then construct a prior latent \ac{gp} $\tilde{g}$ using the simple construction rules presented in Table 1 of \cite{besginow2022constraining}, based on the diagonal entries of $D$. These entries are $0$ or $1$ for controllable systems.
In the case of the system in Equation~\eqref{eq:harmonic_oscillator_diffeq}, we construct the latent \ac{gp} $\tilde{g}$: 
\begin{equation}
    \tilde{g} = \mathcal{GP}\left(\begin{pmatrix}0 \\ 0 \\ 0 \end{pmatrix}, 
        \begin{pmatrix}
            0 & 0 & 0\\
            0 & 0 & 0\\
            0 & 0 & k_\text{SE}
        \end{pmatrix}\right)
    \end{equation}
which we simplify to $g = \mathcal{GP}(\begin{pmatrix}
    0
\end{pmatrix}, \begin{pmatrix}
    k_\text{SE}
\end{pmatrix})$ by removing the independent, uninformative zero-entries in the covariance function of $\tilde{g}$, which deterministically produce zero behavior.
This can be understood as a trivial case of the Gaussian marginalization property, in which we remove constant behavior. 
We simplify $V$ similarly by omitting the first two columns, which correspond to the zeroes in $\tilde{g}$.
By applying the linear operator $V$ to this \emph{latent} \ac{gp} $g$, the realizations of the resulting \ac{lodegp} $Vg$ in Equation~\eqref{eq:LODE_GP_harmonic} \emph{strictly} satisfy the system in Equation~\eqref{eq:harmonic_oscillator_diffeq}, as detailed in \cite{besginow2022constraining}.
By doing so we guarantee that the \ac{lodegp} $Vg$ spans 
the nullspace of $H$, which is equivalent to saying that $Vg$ produces only solutions to the original homogenuous system $H\cdot z=\mathbf{0}$.
\begin{equation}\label{eq:LODE_GP_harmonic}
    Vg = \mathcal{GP}\left(\mathbf{0}, V\cdot 
        \begin{pmatrix}
			k_\text{SE}
		\end{pmatrix}
    \cdot \hat{V}^\top\right)
\end{equation}
where $\hat{V}$ is the operator $V$ applied to the second argument ($t'$) of the \ac{se} covariance function $k_{\text{SE}}$ (cf.\ Equation~\eqref{eq:SE_kernel}).

The \ac{lodegp} in Equation~\eqref{eq:LODE_GP_harmonic} can be trained and conditioned on datapoints, 
which we exploit for our control algorithm.
We illustrate samples drawn from the \ac{lodegp} in Equation~\eqref{eq:LODE_GP_harmonic}, after conditioning it on varying numbers of setpoints, in Figure~\ref{fig:spring_mass_samples}.

\begin{figure}[t]
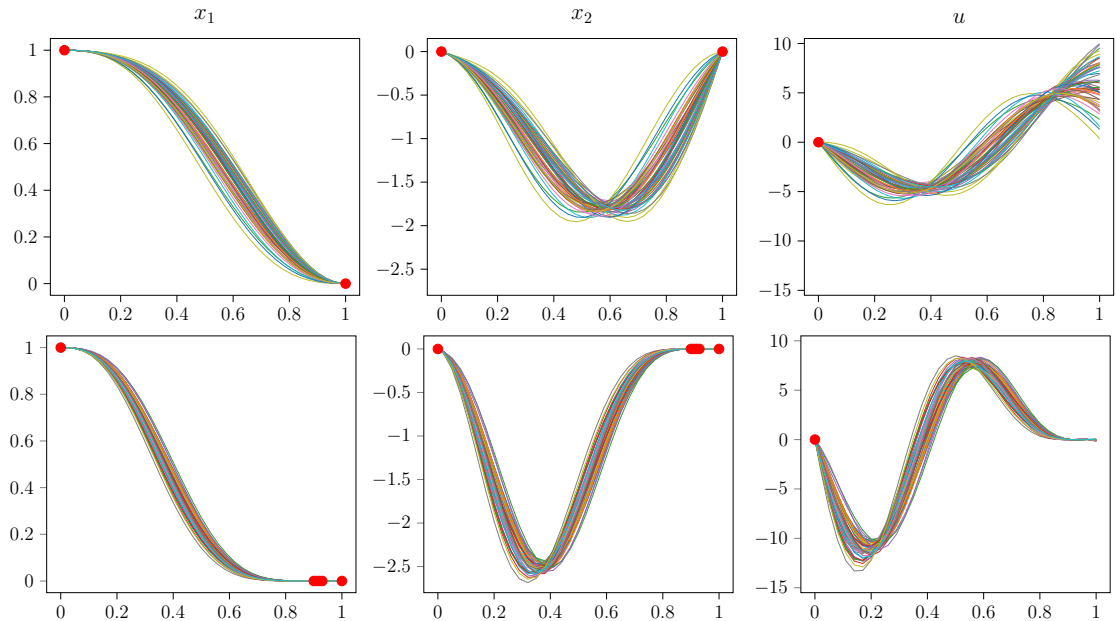

    \centering
    \input{plots/spring_mass_samples.tex}
    \input{plots/5data_spring_mass_samples.tex}
    \caption{
        We see 50 samples drawn from a trained \ac{lodegp} for the spring-mass system in Equation~\eqref{eq:harmonic_oscillator_diffeq}. 
        Each sample is a multivariate time series that spans across all three columns, depicting the distinct channels of the system, and is guaranteed to satisfy the \acp{ode}.
        Setpoints are shown in red and the x-axis and y-axis are shared across rows.
        (Upper) Only the initial state $x_0$ and the desired state $x_T$ are given.
        (Lower) In addition to the initial state and desired state, some intermediate states are given.}
    \label{fig:spring_mass_samples}
\end{figure}

\section{LODE-GP Model Predictive Control}\label{sec:method}

This section describes how to apply our \ac{lodegp} based \ac{mpc} to generate the control input using \ac{gp} conditioning.
We assume a \ac{lodegp} as described in Section~\ref{sec:LODE_GP} for a system of linear \acp{ode}. As we have seen in the previous section, the \ac{lodegp} posterior mean yields smooth functions which satisfy \eqref{eq:mpc:con:ode}. In the following we will show how the remaining constraints \eqref{eq:mpc:con:init} - \eqref{eq:mpc:con:u} can be enforced as hard and soft constraints respectively by specific constructions of the conditioned dataset $\mathcal{D}$ and obtain optimality as in \eqref{eq:mpc:obj}.

\subsection{Controller formulation} \label{sec:method:basic}

We have to respect the initial point constraint~\eqref{eq:mpc:con:init} as hard constraint in timestep $i$. We translate this to the \ac{lodegp} as using $(t_i, z_i)$ in the conditioned dataset $\mathcal{D}$ with noise variance $\sigma_n^2(t_i) = 0 \in \mathbb{R}^{n_z}$ with $n_z = n_x + n_u$. Due to numerical issues, we have to set a numerical jitter of $10^{-8}$ as the noise variance. This forces the \ac{lodegp} posterior mean to satisfy $\mu^*(t_i) = z_i$ up to numerical precision. We define this dataset as $\mathcal{D}_\text{init} = \{(t_i, z_i)\}$ in timestep $i$. 
In order to track the constant reference point $x_\text{ref}$, we choose a constant $z_\text{ref}$ such that Equation~\eqref{eq:harmonic_oscillator_diffeq} is satisfied, which is always possible for controllable systems.
This allows to set the LODE-GP prior mean $\mu(t) = z_\text{ref}$. The posterior mean $\mu^*(t)$ then yields a function which satisfies $\mu^*(t_0) = z_0$ and converges to $z_\text{ref}$ for $t \to \infty$.
As a model predictive control scheme, we update the dataset in each timestep as the new initial point.

\subsection{Open-loop stability}
The controller defined by the \ac{lodegp} and a single reference point yields convergence to the prior mean of the \ac{lodegp}. This results from a property of Bayesian models in general which converge to their prior in the absence of correlated datapoints. This is stated in the following theorem.

\begin{theorem}\label{thm:stability}
    The posterior mean $\mu^*(t)$ of a \ac{lodegp} with $\mathcal{D} = \{(t_0, z_0)\}$ and a squared exponential covariance function converges to its prior for $t \to \pm \infty$.
\end{theorem}
\begin{proof}
    W.l.o.g, $t_0=0$. We have to proof, that $\lim\limits_{t \to \pm \infty} \mu^*(t) = \mu(t)$, which means $K_*K^{-1}(z_0-\mu(t_0) \to 0$.
    $K^{-1}(z_0 - \mu(t_0))$ is independent of $t$, therefore it can be considered as constant. Since
    \begin{equation*}
        K_* = k(t, t_0) = k(t, 0) = \sigma_f^2 p(t) \cdot \exp\left(-\frac{t^2}{2\ell^2}\right)
    \end{equation*}
 for a polynomial $p$ and the exponential function dominates the polynomial for $t \to \pm \infty$, we obtain $K_* \to 0$, which proves the claim.
\end{proof}

\begin{remark}
    The proof of theorem~\ref{thm:stability} provides the convergence rate, which is dependent on the chosen kernel. For the squared exponential kernel we obtain squared exponential decay. This result can be easily extended both for more than one point in dataset $\mathcal{D}$ and for stationary kernels which tend to 0 for $t-t' \to \infty$.
\end{remark}

This theorem suggests, that it is theoretically sufficient to use $\mathcal{D} = \mathcal{D}_\text{init}$ and use the reference $z_\text{ref}$ as constant prior $\mu(t)$ 
to obtain a working controller.
Unfortunately, the thereby generated solution $\mu^*(t)$ is not optimized on controller performance and will not respect state and control constraints.

\subsection{Soft constraints}

In order to fulfill the constraints for state \eqref{eq:mpc:con:x} and control \eqref{eq:mpc:con:u}, we encode pointwise soft constraints at time $t_i$ in the dataset
\begin{align}
    \mathcal{D}_\text{con} &= \{(t_{i+1}, z_\text{con}), \dots,(t_{i+m_c}, z_\text{con})\}, \\
    \text{with } z_\text{con} &= \frac{(z_\text{max} + z_\text{min})}{2} \\
    \text{and } \sigma_n &= \frac{(z_\text{max} - z_\text{min})}{2}
\end{align}
as constraint noise variance. 
Note that $\sigma_n^2 \in \mathbb{R}_{\geq 0}^{n_z}$ allows for different noise levels on each state and control dimension. 
For a point in $\mathcal{D}_\text{con}$ we expect that $z(t_i) = \epsilon_i$ with $\epsilon_i \sim \mathcal{N}(z_\text{con}, \sigma_n^2)$. 
The incorporation of $\mathcal{D}$ imposes soft constraints in the likelihood ~\eqref{eq:gp:MLL} of the \ac{lodegp} and therefore also in the posterior mean $\mu^*(t)$. Note that in timestep $t_i$ we only have constraints on future states. These soft constraints technically do not suggest a uniform constraint satisfaction, but rather focuses the posterior to the center $z_\text{con}$ and suggests a Gaussian distribution with the constraints as $z_\text{con} \pm \sigma_n$.

\begin{algorithm}[h]
    \caption{\ac{lodegp} based \ac{mpc}}
    \KwIn{Initial state $x_{t_0}$ and control $u_{t_0}$, reference $z_\text{ref}$, constraints $\mathcal{D}_\text{con}$ with $\sigma_n^2$}
    \KwOut{Simulation/Control path $\{(x_{t}, u_{t}) | t \in [t_0, t_T]\}$}
    Set prior $\mu(t) = z_\text{ref}$ \\
    Hyperparameter optimization of $\ell^2$ and $\sigma_f^2$ using $\mathcal{D} = \mathcal{D}_\text{init} \cup \mathcal{D}_\text{con}$ \\
    \For {$t=t_0, \dots, t_T$}{
        observe current state $x_{t_i}$ \\
        generate predictive posterior on $[t_i,t_{i+1}]$: \ac{gp} $(x^*_{t_{i+1}}, u^*_{t_{i+1}})$ = $Vg(t_{i+1} | \mathcal{D})$ \\
        set control input for next time interval: $u\vert_{[t_i,t_{i+1}]} =  u^*_{t_{i+1}}$
        }  
    \KwRet{$\{x_{t}, u_{t} | t \in [t_0, t_T]\}$}
    \label{alg:GP_MPC}
\end{algorithm}

\subsection{Optimality in a Reproducing Kernel Hilbert Space} \label{sec:method:opt}
 
Conditioning the \ac{gp} on $\mathcal{D} = \mathcal{D}_\text{init} \cup \mathcal{D}_\text{con}$ yields a posterior mean function $\mu^*(t)$ which satisfies \eqref{eq:mpc:con:ode} and \eqref{eq:mpc:con:init} as hard constraints and \eqref{eq:mpc:con:x} and \eqref{eq:mpc:con:u} as soft constraints. The mean function is optimal in the norm of an abstract space related to the covariance function of the \ac{gp} which is called the \ac{rkhs} \citep{berlinet2011reproducing}. This space is the closure of all possible posterior mean functions with the chosen covariance kernel with respect to its induced norm. For further literature on the connection of \acp{gp} and \acp{rkhs}, see \citep{kanagawa2018gaussian}. The following theorem states this result via the representer theorem \citep{scholkopf2001generalized}.

\begin{theorem}\label{thm:rkhs}
    The posterior mean function as in Equation~\eqref{eq:gaussian_process_posterior_distribution} with noise variance $\sigma_n^2$ for the constraints yields an optimal control minimizing the \ac{rkhs} norm given by the kernel of the LODE-GP.
\end{theorem}
\begin{proof}
     In order to use the representer theorem in
     \citep{scholkopf2001generalized}, we have to provide a positive definite kernel. We use 
     \begin{equation*}
          k(t,t') = V \cdot k_{\text{SE}}(t,t') \cdot \hat{V}^\top + \sigma_n^2 \mathds{1}_{\{t > t_0\}} \cdot \delta(t,t') \cdot \mathds{1}_{\{t' > t_0\}}
     \end{equation*}
     which results in a kernel matrix $K + \text{diag}\left(\sigma_n^2(t_0), \dots,\sigma_n^2(t_{m_c})\right)$. Since the \ac{lodegp} kernel is positive definite, the sum is either which proves the claim.
\end{proof}

This result corresponds to the objective~\eqref{eq:mpc:obj} and provides the optimality of our proposed controller. This completes the translation of the optimization problem described in chapter~\ref{sec:problem} to a problem solved by inference of a \ac{lodegp} with a specific training dataset.
In comparison to \ac{mpc} this optimality is of course limited, since the RKHS is given by the kernel and cannot be chosen as a standard function space. Moreover, the RKHS norm is defined on the extended state $z$ and is not directly able to model $x$ and $u$ separately as it is done in~\eqref{eq:mpc:obj}. In comparison to MPC, the cost function used in our approach is also limited, since it is determined by the cost function formulation in the representer theorem. For further details on the formulation of the RKHS norm for the optimal control problem generated by LODE-GPs, see \cite{besginow2025linear}.

\subsection{Heuristic improvements: Artificial references}
In order to improve the control performance, we propose two heuristic additions to the dataset $\mathcal{D}$.
We propose to extend the conditioned dataset $\mathcal{D}$ with past observed data
\begin{equation}
    \mathcal{D}_\text{past} = \{(t_{i-1}, z_{i-1}, \sigma_n^2(t_{i-1})), \dots,(t_{i-m_p}, z_{i-m_p}, \sigma_n^2(t_{i-m_p}))\}
\end{equation}
which supports the controller in providing smooth behavior. For more information, see Section~\ref{sec:evaluation}.
Further, we suggest virtual reference points to encode hard point constraints for future behavior
\begin{equation}
    \mathcal{D}_\text{v} = \{(t_{j+1}, z_{j+1}, 0),\dots,(t_{m_v}, z_{m_v}, 0)\}.
\end{equation}
This forces the posterior mean to satisfy $\mu^*(t) = z$ for all $(t,z,0) \in \mathcal{D}_\text{v}$. 
Recall that all additional points impose soft constraints in the likelihood of the \ac{lodegp} \eqref{eq:gp:MLL} when conditioning on $\mathcal{D} = \mathcal{D}_\text{init} \cup \mathcal{D}_{\text{con}} \cup \mathcal{D}_\text{v}$.
Note that setpoints may contain masked channels for specific time steps, i.e.\ some $x_i$ or $u_i$ can be missing for timestep $t_i$. %

Figure~\ref{fig:spring_mass_samples} illustrates how the space of functions is constrained by just conditioning the \ac{lodegp} on the intial state $x_0$ and the desired state $x_T$ or by conditioning it on additional setpoints. 
By adding the additional points $x_r$, as in the lower row of Figure~\ref{fig:spring_mass_samples}, we reduce the space of realizations.
This demonstrates, that the dataset $\mathcal{D}_\text{v}$ has to be chosen heuristically with prior knowledge, since virtual datapoints may lead the model to constraint violations and unstable numerical behavior. 
This can be seen in this example, as the control amplitude rises when the state $x_1$ should be regulated faster.

\section{Evaluation}\label{sec:evaluation}

\begin{figure}[t]
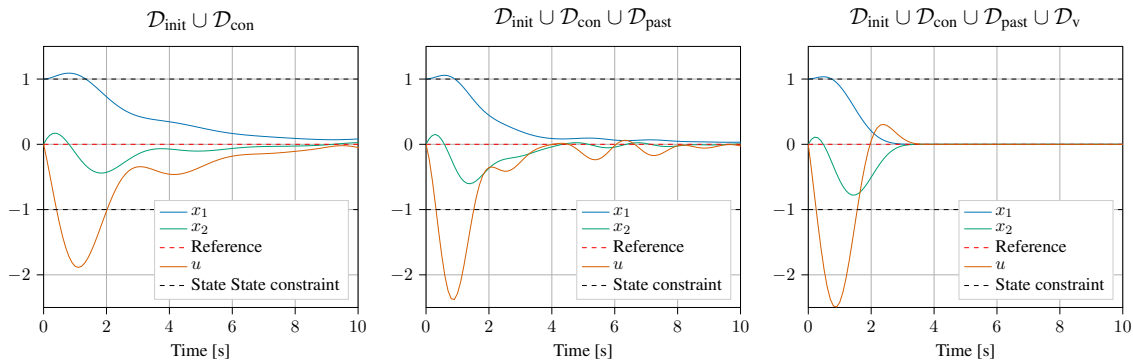

    \centering
    \scalebox{0.61}{\input{plots/model2.tex}}
    \scalebox{0.61}{\input{plots/model3.tex}}
    \scalebox{0.61}{\input{plots/model5.tex}}
    \caption{Comparison of our approach (left) with the two extended models. Each model satisfies the control constraints and the performance improves with larger conditioned dataset $\mathcal{D}$.}
    \label{fig:regulation}
\end{figure}

In this section we present the results of our approach and the proposed extensions on the unstable system introduced in Section~\ref{sec:LODE_GP}.
We investigate the mean constraint violation
\begin{equation}\label{eq:constr_viol}
    \frac{1}{T}\sum_{i=1}^T  \max\{z(t_i) - z_{\max}, 0\} + \max\{z_{\min} - z(t_i), 0\}
\end{equation}
in order to proof whether our approach can handle the imposed constraints.
Moreover, we investigate the mean control error defined as
\begin{equation} \label{eq:control_error}
    \frac{1}{T}\sum_{i=1}^T (x(t_i) - x_{\text{ref}})^2
\end{equation}
in order to compare the control performance of our approach and its extensions.
We compare three different models:

The first model is our approach introduced in \ref{sec:method:basic} -- \ref{sec:method:opt} which uses $\mathcal{D} = \mathcal{D}_\text{init} \cup \mathcal{D}_\text{con}$. The second model is the extension using past data, i.e. $\mathcal{D} = \mathcal{D}_{\text{init}} \cup \mathcal{D}_{\text{con}} \cup \mathcal{D}_\text{past}$. The third model uses additional artificial references as hard constraints. 
The hyperparameters are optimized offline in advance using $\mathcal{D}$. The hyperparameters for model 2 and 3 are thus the same, since the dataset at $t_0$ is identical. The other models have different initial training datasets and thus different hyperparameters.

We investigate a regulation control task, where the state is steered towards the origin from the initial point $x_0 = (1, 0)^\top$ and $u=0$ from $t_0=0$ to $t_T=10$. Note that for our problem formulation the initial control is of importance since the control function has to be smooth on $[t_0,t_T]$. We use $x \in [-1, 1]^2$ and $u \in [-2.5, 2.5]$ as soft constraints. 
For $\mathcal{D}_\text{con}$ we use $100$ equidistant points from $t_1=0.1$ to $t_{100}=10$ with $z_i=0$ and $\sigma_n^2=(1, 1, 2.5)^\top$. For $\mathcal{D}_\text{past}$ we use the last 20 datapoints. For $\mathcal{D}_\text{v}$ we use the same points as in $\mathcal{D}_\text{con}$ with $t > t_v = 4$. This means that for $t > t_v$ the posterior mean is forced to fulfill $\mu^*(t) = 0$. If we set $t_v$ significantly smaller, we obtain constraint violations, since the amplitude of the control would grow. If we set $t_v$ significantly larger, we increase the control error. Therefore the dataset $\mathcal{D}_\text{v}$ has to be chosen carefully.

The results of our experiments are shown in  Table~\ref{tab:regulation} and Figure~\ref{fig:regulation}. They verify, that our controller regulates the unstable system and yields smooth behavior.
We observe, that each model stabilizes the state. The more datapoints are used in $\mathcal{D}$, the faster the system stabilizes, while both extended models have similar constraint and control error. 
Although our approach and the extended model using past observed data have the same hyperparameters, the controller performance is significantly better while adhering to the constraints. Comparing the solution of the initial optimal control problem for $t=t_0$, both models propose the final solution of the model which is extended with past observations. This means, that providing no past measurements changes the behavior of our \ac{lodegp} based controller. Our explanation for this is, that providing no constraints for the past, results in the posterior mean which is most likely to come from the reference $z_\text{ref}$ at $t \to -\infty$. Since the initial point $z_i$ converges towards this point in positive time, the posterior mean changes in each iteration for both $t < t_i$ and $t > t_i$, which results in slower convergence. Constraining the past to observed values by hard therefore keeps the controller on track of the solution proposed in the first optimal control problem. 

\begin{table}[t]
    \centering
     \caption{Results for regulation task.}
    \begin{tabular}{|c|c|c|c|c|}
        \hline
         Training dataset & $\mathcal{D}_\text{init} \cup \mathcal{D}_\text{con}$ & $\mathcal{D}_\text{init} \cup \mathcal{D}_\text{con} \cup \mathcal{D}_\text{past}$ & $\mathcal{D}_\text{init} \cup \mathcal{D}_\text{con} \cup \mathcal{D}_\text{past} \cup \mathcal{D}_\text{v}$  \\
         \hline
    Constraint error \eqref{eq:constr_viol} & 0.0023 & 0.0010 & \textbf{0.0008} \\
    \hline
    Control error \eqref{eq:control_error}& 0.1460 & 0.1066 & \textbf{0.1060} \\
    \hline
    \end{tabular}
   
    \label{tab:regulation}
\end{table}
\section{Conclusion \& Limitations}
\label{sec:conclusion}
We have introduced a novel approach for \ac{mpc} with \acp{gp} acting not as the dynamics model by conditioning on only past data, but as the control policy by conditioning on the initial point and constraints. We have shown the open-loop stability of our controller both theoretically and practically and gave different extensions for the improvement of the control performance. Our controller defined by a \ac{lodegp} posterior mean yields optimal functions in the norm of the corresponding \ac{rkhs}, therefore reducing an optimal control problem to an inference problem. 

Our work shows the possibility to use GPs with its closed form posterior computation as an MPC like controller, although not competing with the vast more practical implementations of MPC. In order to further improve the practicability, for example via guaranteed constraint satisfaction, we suggest to use advanced likelihood formulations relying on variational approaches in order to compute the posterior distribution.

We showed in the last section, that our method solves continuous tracking control problems. The implementation of a discrete LODE-GP is straightforward and the construction of a corresponding Linear \ac{mpc} is analogous and left for future work.
Further research will contain different control strategies given the LODE-GP's distribution of solutions. This includes sampling for safe control paths as done in \citep{tebbe2024efficiently}.

\section*{Acknowledgements}
Andreas Besginow is supported by the research training group ``Dataninja'' (Trustworthy AI for Seamless Problem Solving: Next Generation Intelligence Joins Robust Data Analysis) funded by the German federal state of North Rhine-Westphalia.
Jörn Tebbe and Andreas Besginow are supported by the SAIL project which is funded by the Ministry of Culture and Science of the State of North Rhine-Westphalia under the grant no NW21-059C.
\bibliography{sample}
\end{document}